\documentclass[12 pt, reqno]{amsart}
\usepackage{graphicx}
\usepackage{amsmath}
\usepackage{graphicx}
\usepackage{mathrsfs}

\usepackage[active]{srcltx}
\newtheorem{theorem}{Theorem}[section]

\newtheorem{lemma}[theorem]{Lemma}

\newtheorem{remark}[theorem]{Remark}

\numberwithin{equation}{section}
\let\mathcal=\mathscr

\usepackage{amssymb}
\usepackage{relsize}

= 600pt \footskip=40pt\evensidemargin=10pt
\begin{document}
%\centerline{\textit{Second version: November 29, 2017}}
\title[the strong convergence of a perturbed iterative algorithm]
{On the strong convergence of a perturbed algorithm to the unique solution of a variational inequality problem}%
\author{Ramzi May}%
\address{Mathematics Department, College of Science, King Faisal University, P.O. 380, Ahsaa 31982, Kingdom of Saudi Arabia}
\email{rmay@kfu.edu.sa}
\author{Zahrah BinAli}%
\address{Department of quantitative methods, College of Business, King Faisal University, P.O. 380, Al Ahsaa 31982, Saudi Arabia.}
\email{zbinali@kfu.edu.sa}
%\thanks{Acknowledgement: The authors are grateful to the Deanship of Scientific Research at King Faisal University for
%financially and morally supporting this work under Project ​}

\vskip 0.2cm
\date{December 22, 2021}
\dedicatory{}
% ----------------------------------------------------------------
\begin{abstract}
Let $Q$ be a nonempty closed and convex subset of a real Hilbert space $%
\mathcal{H}$. $T:Q\rightarrow Q$ is a nonexpansive mapping which has a least
one fixed point. $f:Q\rightarrow \mathcal{H}$ is a Lipschitzian function, and $%
F:Q\rightarrow \mathcal{H}$ is a Lipschitzian and strongly monotone mapping.
We prove, under appropriate conditions on the functions $f$ and $F$, the
control real sequences $\{\alpha _{n}\}$ and $\{\beta _{n}\},$ and the error
term $\{e_{n}\},$ that for any starting point $x_{0}$ in $Q,$ the sequence $%
\{x_{n}\}$ generated by the perturbed iterative process
\[
x_{n+1}=\beta _{n}x_{n}+(1-\beta _{n})P_{Q}\left( \alpha
_{n}f(x_{n})+(I-\alpha _{n}F)Tx_{n}+e_{n}\right)
\]%
converges strongly to the unique solution of the variational inequality
problem%
\[
\text{Find }q\in C\text{ such that }\langle F(q)-f(q),x-q\rangle \geq 0\text{
for all }x\in C
\]%
where $C=F_{ix}(T)$ is the set of fixed points of $T.$ Our main result unifies and extends many
well-known previous results.
\end{abstract}
\maketitle
% ----------------------------------------------------------------

\section{Introduction}

Throughout this paper, $\mathcal{H}$ is a real Hilbert space with inner product $\langle
.,.\rangle$ and associated norm $\left\Vert .\right\Vert ,$ $Q$ a nonempty
closed and convex subset of $\mathcal{H},$ $T:Q\rightarrow Q$ is a nonexpansive
mapping (i.e., $\left\Vert Tx-Ty\right\Vert \leq\left\Vert x-y\right\Vert $
for all $x,y\in Q$) such that $C:=F_{ix}(T)=\{x\in Q:Tx=x\}$ is nonempty,
$f:Q\rightarrow\mathcal{H}$ is a Lipschitzian mapping with coefficient
$\alpha\geq0$ (i.e., $\left\Vert f(x)-f(y)\right\Vert \leq\alpha \left\Vert x-y\right\Vert $
for all $x,y\in Q$), and $F:Q\rightarrow\mathcal{H}$ is a Lipschitizian mapping with
coefficient $\kappa>0.$ We assume moreover that $F$ is strongly monotone with
coefficient $\eta>0,$ which means that
\[
\langle F(x)-F(y),x-y\rangle\geq\eta\left\Vert x-y\right\Vert ^{2}\text{ for
all }x,y\in Q.
\]
We assume also that $\alpha<\eta.$ this assumption ensures that the
operator $g:=F-f$ is strongly monotone with coefficient $\eta-\alpha$; which implies that
the variational inequality problem
\begin{equation}
\text{Find }q\in C\text{ such that }\langle F(q)-f(q),x-q\rangle\geq0\text{
for all }x\in C,\tag{VIP}%
\end{equation}
has a unique solution that we denote by $q^{\ast}.$

In the present work, we are concerned with the construction of a general
iterative algorithm that generates sequences converging strongly to $q^{\ast}.$
Let us first recall some previous results related to this subject. In the
particular case $Q=\mathcal{H},$ $f\equiv u$ a constant, and $F=I$ the
identity mapping from $\mathcal{H}$ into itself, Halpern \cite{Hal} introduced the iterative process%
\begin{equation}
\left\{
\begin{array}
[l]{l}%
x_{0}\in\mathcal{H}\\
x_{n+1}=\alpha_{n}u+(1-\alpha_{n})Tx_{n},
\end{array}
\right.  \label{Hal1}%
\end{equation}
with $\{\alpha_{n}\}\in\lbrack0,1].$ He established that if $\alpha_{n}%
=\frac{1}{n^{\theta}}$ with $\theta\in]0,1[$ then the
generated sequence $\{x_{n}\}$ converges strongly to $q^{\ast}$ which is in
this case equal to $P_{C}(u)$ where $P_{C}:\mathcal{H}\rightarrow C$ is the
metric projection from $\mathcal{H}$ onto the closed and convex subset
$C=F_{ix}(T).$ He also proved that the conditions

\begin{enumerate}
\item[(C1)] $\lim_{n\rightarrow+\infty}\alpha_{n}=0,$

\item[(C2)] $\sum_{n=0}^{+\infty}\alpha_{n}=+\infty,$
\end{enumerate}

are necessary for the strong convergence of the algorithm (\ref{Hal1}).
In 1977, Lions \cite{Lio} extended the result of Halpern. In fact, he proved
the strong convergence of sequences $\{x_{n}\}$ generated by process
(\ref{Hal1}) to $q^{\ast}$ provided the sequence $\{\alpha_{n}\}$ satisfies
the necessary conditions (C1)-(C2) and the supplementary condition

\begin{enumerate}
\item[(C3)] $\lim_{n\rightarrow+\infty}\frac{\alpha_{n+1}-\alpha_{n}}%
{\alpha_{n}^{2}}=0.$
\end{enumerate}

In 2000, Moudafi \cite{Mou} considered the case when $Q=\mathcal{H},$
$f:\mathcal{H}\rightarrow\mathcal{H}$ is a contraction with coefficient
$\alpha\in\lbrack0,1[,$ $F=I$ the identity mapping from $\mathcal{H}$ into
itself. He introduced the algorithm%
\begin{equation}
\left\{
\begin{array}
[l]{l}%
x_{0}\in\mathcal{H}\\
x_{n+1}=\alpha_{n}f(x_{n})+(1-\alpha_{n})Tx_{n},
\end{array}
\right.  \label{Mou}%
\end{equation}
where $\{\alpha_{n}\}\in]0,1].$ He established, under the conditions (C1),
(C2) and

\begin{enumerate}
\item[(C4)] $\lim_{n\rightarrow+\infty}\frac{\alpha_{n+1}-\alpha_{n}}%
{\alpha_{n+1}\alpha_{n}}=0,$
\end{enumerate}

\par\noindent the strong convergence of any sequence generated by this algorithm to $q^{\ast
}$ which in this case is equal to the unique fixed point of the contraction
mapping $P_{C}\circ f.$

In 2004, Xu \cite{Xu1} improved Moudafi result; in fact, he followed a new
approach to prove the strong convergence property of the algorithm (\ref{Mou})
provided the sequence $\{\alpha_{n}\}$ satisfies the conditions (C1), (C2) and

\begin{enumerate}
\item[(C5)] $\lim_{n\rightarrow+\infty}\frac{\alpha_{n+1}-\alpha_{n}}%
{\alpha_{n}}=0$ or $\sum_{n=0}^{+\infty}\left\vert \alpha_{n+1}-\alpha
_{n}\right\vert <+\infty.$
\end{enumerate}

Xu \cite{Xu2} has also considered the case when $Q=\mathcal{H},$ $f=u$ a
constant and $F=A$ a $\eta-$ strongly positive, self adjoint, and  bounded linear
operator from $\mathcal{H}$ to $\mathcal{H}.$ He established the strong convergence
of the algorithm%
\[
\left\{
\begin{array}
[l]{l}%
x_{0}\in\mathcal{H}\\
x_{n+1}=\alpha_{n}u+(I-\alpha_{n}A)Tx_{n}%
\end{array}
\right.
\]
to the unique solution $q^{\ast}$ of (VIP) provided the real sequence
$\{\alpha_{n}\}$ satisfies the conditions (C1), (C2) and (C5). Let us notice
here that, in this case, $q^{\ast}$ is the unique minimizer of the strongly
quadratic convex function $J(x)=\frac{1}{2}\langle Ax,x\rangle-\langle u,x\rangle$
over the closed and convex subset $C=F_{ix}(T).$

Later in 2006, Mariano and Xu \cite{MXu} established that the previous strong
convergence result remains true in the more general case when $f:\mathcal{H}%
\rightarrow\mathcal{H}$ is a Lipschitzian mapping with constant $\alpha$
strictly less than $\eta$.

On the other hand, Yamada \cite{Yam} studied the particular case when
$Q=\mathcal{H},$ $f\equiv0.$ He proved that if the sequence $\{\alpha_{n}\}$
satisfies the conditions (C1), (C2) and (C3) then for every starting point
$x_{0}\in\mathcal{H}$ the sequence $\{x_{n}\}$\ generated by the iterative
process
\[
x_{n+1}=(I-\alpha_{n}F)Tx_{n}%
\]
converges strongly to $q^{\ast}.$

In 2010, Tiang \cite{Tia}, by combining the iterative method of Yamada and
the method of Mariano and Xu, had introduced the general algorithm%
\[
\left\{
\begin{array}
[l]{l}%
x_{0}\in\mathcal{H}\\
x_{n+1}=\alpha_{n}f(x_{n})+(I-\alpha_{n}F)Tx_{n}.
\end{array}
\right.
\]
He established the strong convergence of this algorithm to $q^{\ast}$ provided
the real sequence $\{\alpha_{n}\}$ satisfies the conditions (C1), (C2)
and (C5).

Later, in 2011, Ceng, Ansari and Yao \cite{CAY} extended Tiang's result to the case
where $Q$ is not necessary equal to the whole space $\mathcal{H}.$ Precisely,
they proved that if the sequence $\{\alpha_{n}\}$ satisfies the conditions
(C1), (C2) and (C5), then for any starting point $x_{0}$ in $Q$ the sequence
$\{x_{n}\}$ defined by the scheme%
\[
x_{n+1}=P_{Q}\left(  \alpha_{n}f(x_{n})+(I-\alpha_{n}F)Tx_{n}\right)
\]
converges strongly to $q^{\ast}.$

In this paper, inspired by the previous works and the papers
\cite{YLC} and \cite{BCM}, we introduce the following hybrid and perturbed
algorithm:%
\begin{equation}
\left\{
\begin{array}
[l]{l}%
x_{0}\in Q\\
x_{n+1}=\beta_{n}x_{n}+(1-\beta_{n})P_{Q}\left(  \alpha_{n}f(x_{n}%
)+(I-\alpha_{n}F)Tx_{n}+e_{n}\right)  ,
\end{array}
\right.  \tag{HPA}%
\end{equation}
where $\{\alpha_{n}\}$ and $\{\beta_{n}\}$ are two real sequences in $[0,1]$
and $\{e_{n}\}$ is a sequence in $\mathcal{H}$ representing the perturbation
term. Roughly speaking, we will prove that any
sequence $\{x_{n}\}$ generated by the algorithm (HPA) converges strongly to
$q^{\ast}$ provided that the sequence $\{\alpha_{n}\}$
satisfies only the necessary conditions (C1) and (C2), the sequence
$\{\beta_{n}\}$ is not too close to $0$ or $1,$ and the perturbation term
$\{e_{n}\}$ is relatively small with respect to $\{\alpha_{n}\}$.

The paper is organized as follows. In the next section, we recall some preliminary lemmas that will be used frequently in the proof of the results of the paper. The section 3 is devoted to the study of the convergence of an implicit version of the algorithm (HPA). The strong convergence of the iterative algorithm (HPA) will be investigated in Section4. The Last section will be devoted to the study of the limit case where the strong monotonicity coefficient of $F$ is equal to the Lipschitzian coefficient of $f$.

\section{Preliminaries}

In this section, we recall some classical results that will be useful in the
proof of the main theorems of the paper.

The first result is a simple but powerful lemma proved by Xu in \cite{Xu}.
This lemma is a generalization of a result due to Bertsekas (see \cite[Lemma 1.5.1]{Ber}).

\begin{lemma}
\label{Lem1}let $\{a_{n}\}$ be a sequence of nonnegative real numbers such
that:%
\[
a_{n+1}\leq(1-\gamma_{n})a_{n}+\gamma_{n}r_{n}+\delta_{n},\text{ }n\geq0,
\]
where $\{\gamma_{n}\}\in\lbrack0,1]$ and $\{r_{n}\}$and $\{\delta_{n}\}$ are
two real sequences such that

\begin{enumerate}
\item[(1)] $\sum_{n=0}^{+\infty}\gamma_{n}=+\infty;$

\item[(2)] $\sum_{n=0}^{+\infty}\left\vert \delta_{n}\right\vert <+\infty; $

\item[(3)] $\lim\sup_{n\rightarrow+\infty}r_{n}\leq0.$
\end{enumerate}
Then the sequence $\{a_{n}\}$ converges to $0.$
\end{lemma}
The second result is the following lemma due to Suzuki \cite{Suz}
\begin{lemma}
\label{Lem2}Let $\{z_{n}\}$ and $\{w_{n}\}$ be two bounded sequence in a
Banach space $E$ and let $\{\beta_{n}\}$ be a sequence in $[0,1]$ with
\[
0<\lim\inf_{n\rightarrow+\infty}\beta_{n}\leq\lim\sup_{n\rightarrow+\infty
}\beta_{n}<1.
\]
Suppose that
\[
z_{n+1}=\beta_{n}z_{n}+(1-\beta_{n})w_{n},\text{ }n\geq0
\]
and
\[
\lim\sup_{n\rightarrow+\infty}\left(  \left\Vert w_{n+1}-w_{n}\right\Vert
-\left\Vert z_{n+1}-z_{n}\right\Vert \right)  \leq0.
\]
Then
\[
\lim_{n\rightarrow+\infty}\left\Vert z_{n}-w_{n}\right\Vert =0.
\]

\end{lemma}

The last result of this section is a particular case of the well-known
demiclosedness principle (see \cite[Corollary 4.18]{BC}).

\begin{lemma}
\label{Lem3} Let $\{x_{n}\}$ be a sequence in $Q$. If $\{x_{n}\}$ converges weakly to some $x$
and $\{x_{n}-Tx_{n}\}$ converges strongly to $0$, then $x\in
F_{ix}(T).$
\end{lemma}

\section{The convergence of an implicit version of the algorithm (HPA)}

In this section, we prove the strong convergence of the perturbed and implicit algorithm%

\[
x_{t}=P_{Q}(tf(x_{t})+(I-tF)Tx_{t}+e(t))
\]
as $t\rightarrow0^{+}$ to the unique solution $q^{\ast}$ of the variational inequality
problem (VIP) provided that the perturbation term $e(t)$ is sufficiently small.
More precisely, we will prove the following theorem.

\begin{theorem}
\label{The1}let $\delta_{0}^{\ast}:=2\frac{\eta-\alpha}{\kappa^{2}}$ and
$e:]0,\delta_{0}^{\ast}[\rightarrow\mathcal{H}$ such that
\[
\lim_{t\rightarrow0^{+}}\frac{\left\Vert e(t)\right\Vert }{t}=0.
\]
Then for every $t\in]0,\delta_{0}^{\ast}[$ there exists a unique $x_{t}\in Q$
such that%
\[
x_{t}=P_{Q}(tf(x_{t})+(I-tF)Tx_{t}+e(t)).
\]
Moreover, $x_{t}$ converges strongly in $\mathcal{H}$ as $t\rightarrow0^{+}$
toward $q^{\ast}$ the unique solution of the variational inequality problem (VIP).
\end{theorem}

The proof relies essentially on the following lemma which will be also used in
the next section devoted to the study of the strong convergence of the
algorithm (HPA).

\begin{lemma}
\label{Lem4}Let $\delta_{0}\in]0,\delta_{0}^{\ast}[.$ For every $t\in
]0,\delta_{0}],$ the mapping $S_{t}:Q\rightarrow\mathcal{H}$ defined by
\[
S_{t}(x)=tf(x)+(I-tF)Tx
\]
is Lipschitzian with coefficient $1-t\sigma_{0}$ where $\sigma_{0}%
:=\eta-\alpha-\frac{\kappa^{2}\delta_{0}}{2}.$
\end{lemma}

\begin{proof}
Let $t\in]0,\delta_{0}]$ and $x,y\in Q.$ We have
\begin{align*}
\left\Vert (I-tF)Tx-(I-tF)Ty\right\Vert ^{2}  &  =\left\Vert Tx-Ty\right\Vert
^{2}-2t\langle F(Tx)-F(Ty),Tx-Ty\rangle\\
&+ t^{2}\left\Vert F(Tx)-F(Ty)\right\Vert
^{2}\\
&  \leq\left(  1-2t\eta+t^{2}\kappa^{2}\right)  \left\Vert Tx-Ty\right\Vert
^{2}\\
&  \leq\left(  1-2t(\eta-\frac{t\kappa^{2}}{2})\right)  \left\Vert
x-y\right\Vert ^{2}.
\end{align*}
Hence by using the elementary inequality%
\[
\sqrt{1-x}\leq1-\frac{x}{2},\text{ for all }x\in\lbrack0,1],
\]
we deduce that
\[
\left\Vert (I-tF)Tx-(I-tF)Ty\right\Vert \leq\left(  1-t(\eta-\frac{t\kappa
^{2}}{2})\right)  \left\Vert x-y\right\Vert .
\]
Therefore,%
\begin{align*}
\left\Vert S_{t}(x)-S_{t}(y)\right\Vert  &  \leq t\left\Vert
f(x)-f(y)\right\Vert +\left\Vert (I-tF)Tx-(I-tF)Ty\right\Vert \\
&  \leq\left(  t\alpha+1-t(\eta-\frac{t\kappa^{2}}{2})\right)  \left\Vert
x-y\right\Vert \\
&  =\left(  1-t(\eta-\alpha-\frac{t\kappa^{2}}{2})\right)  \left\Vert
x-y\right\Vert \\
&  \leq\left(  1-t(\eta-\alpha-\frac{\kappa^{2}\delta_{0}}{2})\right)
\left\Vert x-y\right\Vert \\
&  =(1-\sigma_{0}t)\left\Vert x-y\right\Vert .
\end{align*}
This completes the proof.
\end{proof}

Now we are in position to prove the main result of this section.

\begin{proof} Let $\delta_{0}\in]0,\delta_{0}^{\ast}[$ be a fixed
real. Let $t\in]0,\delta_{0}].$ Since the operator $P_{Q}$ is nonexpansive, it
follows from the previous lemma that the two mapping $\varphi_{t}$ and
$\phi_{t}$ defined from $Q$ to $Q$ by%
\begin{align*}
\varphi_{t}(x) &  =P_{Q}\left(  S_{t}(x)\right)  ,\\
\phi_{t}(x) &  =P_{Q}\left(  S_{t}(x)+e(t)\right),
\end{align*}
are contractions with the same coefficient $1-t\sigma_{0}\in\lbrack0,1[.$
Hence, the classical Banach fixed point theorem ensures the existence of a
unique $x_{t}$ and $y_{t}$ in $Q$ such that $x_{t}=P_{Q}\left(  S_{t}%
(x_{t})\right)  $ and $y_{t}=P_{Q}\left(  S_{t}(y_{t})+e(t)\right).$ Using again Lemma \ref{Lem4} and the fact that $P_{Q}$ is nonexpansive, we
obtain
\[
\left\Vert x_{t}-y_{t}\right\Vert \leq(1-t\sigma_{0})\left\Vert x_{t}%
-y_{t}\right\Vert +\left\Vert e(t)\right\Vert ,
\]
which implies that%
\[
\left\Vert x_{t}-y_{t}\right\Vert \leq\frac{\left\Vert e(t)\right\Vert
}{t\sigma_{0}}.
\]
Hence, from the assumption on $e(t),$ we get
\[
\left\Vert x_{t}-y_{t}\right\Vert \rightarrow0\text{ as }t\rightarrow0^{+}.
\]
Therefore, in order to prove that $x_{t}\rightarrow q^{\ast}$ as $t\rightarrow
0^{+}$, it suffices to prove that $y_{t}\rightarrow q^{\ast}$ as
$t\rightarrow0^{+}.$

To do this let us first show that the family $(y_{t})_{0<t\leq\delta_{0}}$ is
bounded in $\mathcal{H}.$ Pick $q\in F_{ix}(T).$ By using the fact that
$P_{Q}$ is nonexpansive and Lemma \ref{Lem4}, we easily deduce that for every
$t\in]0,\delta_{0}]$ we have%
\begin{align*}
\left\Vert y_{t}-q\right\Vert  &  =\left\Vert y_{t}-P_{Q}(q)\right\Vert \\
&  \leq\left\Vert P_{Q}(S_{t}(y_{t}))-P_{Q}(S_{t}(q))\right\Vert +\left\Vert
P_{Q}(S_{t}(q))-P_{Q}(q)\right\Vert \\
&  \leq\left\Vert S_{t}(y_{t})-S_{t}(q)\right\Vert +\left\Vert S_{t}%
(q)-q\right\Vert \\
&  \leq(1-t\sigma_{0})\left\Vert y_{t}-q\right\Vert +t\left\Vert
f(q)-F(q)\right\Vert .
\end{align*}
Hence,%
\[
\sup_{0<t\leq\delta_{0}}\left\Vert y_{t}-q\right\Vert \leq\frac{\left\Vert
f(q)-F(q)\right\Vert }{\sigma_{0}},
\]
which implies that $(y_{t})_{0<t\leq\delta_{0}}$ is bounded in $\mathcal{H}$,
and so is $\left(  f(y_{t})-F(Ty_{t})\right)  _{0<t\leq\delta_{0}}$ since the
mapping $f-F\circ T$ is Lipschitzian. Therefore, there exists a constant $M>0$
such that for every $t\in]0,\delta_{0}]$ we have%
\begin{align*}
\left\Vert y_{t}-Ty_{t}\right\Vert  &  =\left\Vert P_{Q}(S(y_{t}%
))-P_{Q}(Ty_{t})\right\Vert \\
&  \leq\left\Vert S_{t}(y_{t})-Ty_{t}\right\Vert \\
&  =t\left\Vert f(y_{t})-F(Ty_{t})\right\Vert \\
&  \leq M~t.
\end{align*}
Hence,%
\begin{equation}
y_{t}-Ty_{t}\rightarrow0\text{ in }\mathcal{H}\text{ as }t\rightarrow
0^{+}.\label{h1}%
\end{equation}
On the other hand, since $(y_{t})_{0<t\leq\delta_{0}}$ is bounded in
$\mathcal{H},$ there exists a sequence $(t_{n})_{n}\in]0,\delta_{0}]$ which
converges to $0$ such that the sequence $\{y_{t_{n}}\}$ converges weakly in
$\mathcal{H}$ to some $y$ and
\begin{align*}
\lim\sup_{t\rightarrow0^{+}}\langle y_{t}-q^{\ast},f(q^{\ast})-F(q^{\ast
})\rangle &  =\lim_{n\rightarrow+\infty}\langle y_{t_{n}}-q^{\ast},f(q^{\ast
})-F(q^{\ast})\rangle\\
&  =\langle y-q^{\ast},f(q^{\ast})-F(q^{\ast})\rangle.
\end{align*}
Let us notice that, thanks to Lemma \ref{Lem3}, we deduce from (\ref{h1}) that
$y\in F_{ix}(T);$ hence, from the definition of $q^{\ast},$ we conclude that
\begin{equation}
\lim\sup_{t\rightarrow0^{+}}\langle y_{t}-q^{\ast},f(q^{\ast})-F(q^{\ast
})\rangle\leq0.\label{h2}%
\end{equation}
Finally, for every $t\in]0,\delta_{0}],$ we have%
\begin{align*}
\left\Vert y_{t}-q^{\ast}\right\Vert ^{2} &  =\left\Vert P_{Q}(S_{t}%
(y_{t}))-P_{Q}(q^{\ast})\right\Vert ^{2}\\
&  \leq\left\Vert S_{t}(y_{t})-q^{\ast}\right\Vert ^{2}\\
&  =\left\Vert S_{t}(y_{t})-S_{t}(q^{\ast})\right\Vert ^{2}+2\langle
S_{t}(y_{t})-S_{t}(q^{\ast}),S_{t}(q^{\ast})-q^{\ast}\rangle+\left\Vert
S_{t}(q^{\ast})-q^{\ast}\right\Vert ^{2}\\
&  \leq(1-t\sigma_{0})^{2}\left\Vert y_{t}-q^{\ast}\right\Vert ^{2}+2t\langle
y_{t}-q^{\ast},f(q^{\ast})-F(q^{\ast})\rangle\\
&  +2t^{2}\langle f(y_{t}%
)-F(Ty_{t}),f(q^{\ast})-F(q^{\ast})\rangle -t^{2}\left\Vert f(q^{\ast})-F(q^{\ast})\right\Vert ^{2}\\
&  \leq\left(  1-2\sigma_{0}t\right)  \left\Vert y_{t}-q^{\ast}\right\Vert
^{2}+2t~\langle y_{t}-q^{\ast},f(q^{\ast})-F(q^{\ast})\rangle+Ct^{2},
\end{align*}
where $C>0$ is a constant independent of $n.$ Therefore for every
$t\in]0,\delta_{0}]$%
\[
\left\Vert y_{t}-q^{\ast}\right\Vert ^{2}\leq\frac{1}{\sigma_{0}}\left(
\langle y_{t}-q^{\ast},f(q^{\ast})-F(q^{\ast})\rangle+\frac{C}{2}t\right)
\]
Thus, by using the estimate (\ref{h2}), we deduce that $y_{t}\rightarrow
q^{\ast}$ in $\mathcal{H}$ as $t\rightarrow0^{+}.$ This completes the proof of
Theorem \ref{The1}.
\end{proof}

\section{The convergence of the algorithm (HPA).}

In this section, we study the strong convergence of the averaged and perturbed algorithm (HPA)%

\begin{equation}
x_{n+1}=\beta_{n}x_{n}+(1-\beta_{n})P_{Q}\left(  \alpha_{n}f(x_{n}%
)+(I-\alpha_{n}F)Tx_{n}+e_{n}.\right)  \tag{HPA}%
\end{equation}
Precisely, we will prove the following result.

\begin{theorem}
\label{The2}Let $\{e_{n}\}$ be a sequence in $\mathcal{H}$ and $\{\alpha
_{n}\}\in]0,1]$ and $\{\beta_{n}\}\in\lbrack0,1]$ two real sequences such that:

\begin{enumerate}
\item[(i)] $\alpha_{n}\rightarrow0$ and $\sum_{n=0}^{+\infty}\alpha
_{n}=+\infty$

\item[(ii)] One of the two following two conditions is satisfied:

\item[(h1)] $0<\lim\inf_{n\rightarrow+\infty}\beta_{n}\leq\lim\sup
_{n\rightarrow+\infty}\beta_{n}<1.$

\item[(h2)] $\lim\sup_{n\rightarrow+\infty}\beta_{n}<1$, either $\frac{\beta
_{n+1}-\beta_{n}}{\alpha_{n}}\rightarrow0$ or $\sum_{n=0}^{+\infty}\left\vert
\beta_{n+1}-\beta_{n}\right\vert <\infty$ and either $\frac{\alpha_{n+1}-\alpha_{n}%
}{\alpha_{n}}\rightarrow1$ or $\sum_{n=0}^{+\infty}\left\vert \alpha
_{n+1}-\alpha_{n}\right\vert <\infty.$

\item[(iii)] $\sum_{n=0}^{+\infty}\left\Vert e_{n}\right\Vert <\infty$ or
$\frac{\left\Vert e_{n}\right\Vert }{\alpha_{n}}\rightarrow0.$
\end{enumerate}
Then for every $x_{0}\in Q$, the sequence $\{x_{n}\}$ generated
by the algorithm (HPA) converges strongly in $\mathcal{H}$ to $q^{\ast}$ the
unique solution of the variational inequality problem (VIP).
\end{theorem}
\begin{proof}
Since $\alpha_{n}\rightarrow0$ $as$ $n\rightarrow\infty$ and 
we are only interested in the study of the asymptotic behavior of the
sequence $\{x_{n}\}$, we can assume without loss of generality that for all $n\in\mathbb{N},$
$\alpha_{n}\in]0,\delta_{0}]$ where $\delta_{0}\in]0,\delta_{0}^{\ast}[$ is a
fixed real. Let $\{y_{n}\}$ be the sequence defined as follows%
\[
\left\{
\begin{array}
[l]{l}%
y_{0}=x_{0}\\
y_{n+1}=\beta_{n}y_{n}+(1-\beta_{n})P_{Q}(\alpha_{n}f(y_{n})+(I-\alpha
_{n}F)Ty_{n}),~n\geq0.
\end{array}
\right.
\]
Using the fact $P_{Q}$ is nonexpansive and Lemma \ref{Lem4}, we easily obtain%
\begin{align*}
\left\Vert y_{n+1}-x_{n+1}\right\Vert  &  \leq\beta_{n}\left\Vert y_{n}%
-x_{n}\right\Vert +(1-\beta_{n})\left\Vert P_{Q}(S_{\alpha_{n}}(y_{n}%
))-P_{Q}(S_{\alpha_{n}}(x_{n})+e_{n})\right\Vert \\
&  \leq[\beta_{n}+(1-\beta_{n})\left(  1-\sigma_{0}\alpha_{n}\right)
]\left\Vert y_{n}-x_{n}\right\Vert +(1-\beta_{n})\left\Vert e_{n}\right\Vert
\\
&  \leq(1-\gamma_{n})\left\Vert y_{n}-x_{n}\right\Vert +\left\Vert
e_{n}\right\Vert ,
\end{align*}
where $\gamma_{n}=\sigma_{0}(1-\beta_{n})\alpha_{n}.$ 
\par\noindent Since $\lim
\sup_{n\rightarrow+\infty}\beta_{n}<1,$ there exists $a>0$ and $n_{0}%
\in\mathbb{N}$ such that $a\alpha_{n}\leq\gamma_{n}\leq1$ for all $n\geq
n_{0}$. Hence, by applying Lemma \ref{Lem1}, we deduce that%
\begin{equation}
y_{n}-x_{n}\rightarrow0. \label{kk}%
\end{equation}
Therefore it suffices to prove that the sequence $\{y_{n}\}$ converges
strongly to $q^{\ast}$ to conclude that $\{x_{n}\}$ also converges strongly to
$q^{\ast}.$

\par\noindent Let us first show that $\{y_{n}\}$ is bounded in $\mathcal{H}.$ Let $q\in
F_{ix}(T).$ For every $n\in\mathbb{N},$ we have%
\begin{align*}
\left\Vert y_{n+1}-q\right\Vert  &  \leq\beta_{n}\left\Vert y_{n}-q\right\Vert
+(1-\beta_{n})[\left\Vert P_{Q}(S_{\alpha_{n}}(y_{n}))-P_{Q}(S_{\alpha_{n}%
}(q))\right\Vert \\
&  +\left\Vert P_{Q}(S_{\alpha_{n}}(q))-P_{Q}(q)\right\Vert ]\\
&  \leq\beta_{n}\left\Vert y_{n}-q\right\Vert +(1-\beta_{n})[\left\Vert
S_{\alpha_{n}}(y_{n})-S_{\alpha_{n}}(q)\right\Vert +\left\Vert S_{\alpha_{n}%
}(q)-q\right\Vert ]\\
&  \leq\beta_{n}\left\Vert y_{n}-q\right\Vert +(1-\beta_{n})\left[  \left(
1-\sigma_{0}\alpha_{n}\right)  \left\Vert y_{n}-q\right\Vert +\alpha
_{n}\left\Vert f(q)-F(q)\right\Vert \right]  .
\end{align*}
The last inequality immediately implies that the sequence
$$v_{n}%
:=\max\{\left\Vert y_{n}-q\right\Vert ,\frac{\left\Vert f(q)-F(q)\right\Vert
}{\sigma_{0}}\}$$
is decreasing. Therefore the sequence $\{y_{n}\}$ is bounded
in $\mathcal{H}$ and so are the sequences $\{f(y_{n})\}$ and $\{F(Ty_{n})\}.$
\par\noindent Now we are going to prove that
\begin{equation}
y_{n}-Ty_{n}\rightarrow0.\label{HM}%
\end{equation}
Let us first assume that the condition (h1) is satisfied. For every
$n\in\mathbb{N},$ set $z_{n}=P_{Q}\left(  \alpha_{n}f(y_{n})+(I-\alpha
_{n}F)Ty_{n}\right)  .$ We have, the sequences $\{y_{n}\}$ and $\{z_{n}\}$ are
bounded in $\mathcal{H},$
\[
y_{n+1}=\beta_{n}y_{n}+(1-\beta_{n})z_{n}\text{ for every }n,
\]
and
\begin{align*}
\left\Vert z_{n+1}-z_{n}\right\Vert  &  \leq(\alpha_{n}+\alpha_{n+1}%
)\sup_{m\geq0}\left\Vert f(y_{m})-F(Ty_{m})\right\Vert +\left\Vert
Ty_{n+1}-Ty_{n}\right\Vert \\
&  \leq(\alpha_{n}+\alpha_{n+1})\sup_{m\geq0}\left\Vert f(y_{m})-F(Ty_{m}%
)\right\Vert +\left\Vert y_{n+1}-y_{n}\right\Vert
\end{align*}
which implies
\[
\lim\sup_{n\rightarrow+\infty}\left\Vert z_{n+1}-z_{n}\right\Vert -\left\Vert
y_{n+1}-y_{n}\right\Vert \leq0.
\]
Therefore, from Lemma \ref{Lem2}, we deduce that%
\[
z_{n}-y_{n}\rightarrow0,
\]
which combined with the fact that%
\begin{align*}
\left\Vert z_{n}-Ty_{n}\right\Vert  &  =\left\Vert P_{Q}\left(  \alpha
_{n}f(y_{n})+(I-\alpha_{n}F)Ty_{n}\right)  -P_{Q}(Ty_{n})\right\Vert \\
&  \leq\alpha_{n}\sup_{m\geq0}\left\Vert f(y_{m})-F(Ty_{m})\right\Vert
\rightarrow0\text{ as }n\rightarrow+\infty,
\end{align*}
implies the required result (\ref{HM}).

Let us now establish (\ref{HM}) under the assumption (h2). A simple
computation using the fact that the sequence $\{y_{n}\}$ and $\{P_{Q}\left(
\alpha_{n}f(y_{n})+(I-\alpha_{n}F)Ty_{n}\right)  \}$ are bounded in
$\mathcal{H},$ ensures the existence of two real constants $M_{1},M_{2}>0$
such that for every $n\in\mathbb{N},$%
\begin{align*}
\left\Vert y_{n+1}-y_{n}\right\Vert  &  \leq\beta_{n}\left\Vert y_{n}%
-y_{n-1}\right\Vert +(1-\beta_{n})\left\Vert P_{Q}(S_{\alpha_{n}}%
(y_{n}))-P_{Q}(S_{\alpha_{n-1}}(y_{n-1}))\right\Vert +M_{1}\left\vert
\beta_{n}-\beta_{n-1}\right\vert \\
&  \leq\beta_{n}\left\Vert y_{n}-y_{n-1}\right\Vert +(1-\beta_{n})\left\Vert
S_{\alpha_{n}}(y_{n})-S_{\alpha_{n-1}}(y_{n-1})\right\Vert +M_{1}\left\vert
\beta_{n}-\beta_{n-1}\right\vert \\
&  \leq\beta_{n}\left\Vert y_{n}-y_{n-1}\right\Vert +(1-\beta_{n})\left\Vert
S_{\alpha_{n}}(y_{n})-S_{\alpha_{n}}(y_{n-1})\right\Vert +\\
&  (1-\beta_{n})\left\Vert S_{\alpha_{n}}(y_{n-1})-S_{\alpha_{n-1}}%
(y_{n-1})\right\Vert +M_{1}\left\vert \beta_{n}-\beta_{n-1}\right\vert \\
&  \leq\beta_{n}\left\Vert y_{n}-y_{n-1}\right\Vert +(1-\beta_{n}%
)(1-\sigma_{0}\alpha_{n})\left\Vert y_{n}-y_{n-1}\right\Vert\\
&  + M_{1}\left\vert \beta_{n}-\beta_{n-1}\right\vert +M_{2}\left\vert
\alpha_{n}-\alpha_{n-1}\right\vert \\
&  =(1-\sigma_{0}(1-\beta_{n})\alpha_{n})\left\Vert y_{n}-y_{n-1}\right\Vert
+M_{1}\left\vert \beta_{n}-\beta_{n-1}\right\vert +M_{2}\left\vert \alpha
_{n}-\alpha_{n-1}\right\vert .
\end{align*}
Hence, by proceeding as in the proof of (\ref{kk}), we infer that
\begin{equation}
\left\Vert y_{n+1}-y_{n}\right\Vert \rightarrow0. \label{RN}%
\end{equation}
On the other hand, for every $n\in\mathbb{N},$ we have%
\begin{align*}
\left\Vert y_{n+1}-Ty_{n}\right\Vert  &  \leq\beta_{n}\left\Vert y_{n}%
-Ty_{n}\right\Vert +(1-\beta_{n})\left\Vert P_{Q}(S_{\alpha_{n}}(y_{n}%
))-P_{Q}(Ty_{n})\right\Vert \\
&  \leq\beta_{n}\left\Vert y_{n+1}-Ty_{n}\right\Vert +\beta_n\left\Vert y_{n+1}-y_{n}\right\Vert+\left\Vert S_{\alpha_{n}}(y_{n}%
)-Ty_{n}\right\Vert\\
&  \leq\beta_{n}\left\Vert y_{n+1}-Ty_{n}\right\Vert +\left\Vert y_{n+1}-y_{n}\right\Vert+\alpha_{n}\left\Vert
f(y_{n})-F(Ty_{n})\right\Vert.
\end{align*}
Hence, we obtain the inequality%
\[
\left\Vert y_{n+1}-Ty_{n}\right\Vert \leq\frac{1}{1-\beta_{n}}\left(
\left\Vert y_{n+1}-y_{n}\right\Vert +\alpha_{n}\sup_{m\geq0}\left\Vert
f(y_{m})-F(Ty_{m})\right\Vert \right)  ,
\]
which, combined with (\ref{RN}) and the fact that $\lim\sup_{n\rightarrow
+\infty}\beta_{n}<1$, implies that%
\begin{equation}
\left\Vert y_{n+1}-Ty_{n}\right\Vert \rightarrow0. \label{RR}%
\end{equation}
The required estimate (\ref{HM}), follows from (\ref{RN}) and (\ref{RR}).

Now we are going to apply the fundamental Lemma \ref{Lem1} to conclude that
$\{y_{n}\}$ converges strongly to $q^{\ast}.$ But first let us notice that by
proceeding as in the proof of Theorem \ref{The1} and by using (\ref{HM}) and
the fact that $\{y_{n}\}$ is bounded, we deduce that%

\[
\lim\sup_{n\rightarrow+\infty}\langle y_{n}-q^{\ast},f(q^{\ast})-F(q^{\ast
})\rangle\leq0.
\]
Combined with the estimate (\ref{HM}), the last inequality yields%
\begin{equation}
\lim\sup_{n\rightarrow+\infty}\langle Ty_{n}-q^{\ast},f(q^{\ast})-F(q^{\ast
})\rangle\leq0.\label{HK}%
\end{equation}
Finally, for every $n\in\mathbb{N},$%
\begin{align*}
\left\Vert y_{n+1}-q^{\ast}\right\Vert ^{2} &  \leq\beta_{n}\left\Vert
y_{n}-q^{\ast}\right\Vert ^{2}+(1-\beta_{n})\left\Vert P_{Q}(S_{\alpha_{n}%
}(y_{n}))-P_{Q}(q^{\ast})\right\Vert ^{2}\\
&  \leq\beta_{n}\left\Vert y_{n}-q^{\ast}\right\Vert ^{2}+(1-\beta
_{n})\left\Vert S_{\alpha_{n}}(y_{n}))-q^{\ast}\right\Vert ^{2}\\
&  =\beta_{n}\left\Vert y_{n}-q^{\ast}\right\Vert ^{2}+(1-\beta_{n}%
)[\left\Vert S_{\alpha_{n}}(y_{n}))-S_{\alpha_{n}}(q^{\ast})\right\Vert ^{2}\\
&  +2\langle S_{\alpha_{n}}(y_{n}))-S_{\alpha_{n}}(q^{\ast}),S_{\alpha_{n}%
}(q^{\ast})-q^{\ast}\rangle+\left\Vert S_{\alpha_{n}}(q^{\ast})-q^{\ast
}\right\Vert ^{2}]\\
&  =\beta_{n}\left\Vert y_{n}-q^{\ast}\right\Vert ^{2}+(1-\beta_{n}%
)[\left\Vert S_{\alpha_{n}}(y_{n}))-S_{\alpha_{n}}(q^{\ast})\right\Vert ^{2}\\
&  +2\langle S_{\alpha_{n}}(y_{n}))-q^{\ast},S_{\alpha_{n}}(q^{\ast})-q^{\ast
}\rangle-\left\Vert S_{\alpha_{n}}(q^{\ast})-q^{\ast}\right\Vert ^{2}]\\
&  \leq\beta_{n}\left\Vert y_{n}-q^{\ast}\right\Vert ^{2}+
(1-\beta_{n})[(1-\alpha_{n}\sigma_{0})^{2}\left\Vert y_{n}-q^{\ast}\right\Vert
^{2}\\
&  +2\alpha_{n}^{2}\langle f(y_{n})-F(Ty_{n}),f(q^{\ast})-F(q^{\ast}%
)\rangle+2\alpha_{n}\langle Ty_{n}-q^{\ast},f(q^{\ast})-F(q^{\ast})\rangle]\\
&  \leq(1-\gamma_{n})\left\Vert y_{n}-q^{\ast}\right\Vert ^{2}+2\alpha
_{n}(1-\beta_{n})\left[  \langle Ty_{n}-q^{\ast},f(q^{\ast})-F(q^{\ast
})\rangle+C\alpha_{n}\right]  \\
&  =(1-\gamma_{n})\left\Vert y_{n}-q^{\ast}\right\Vert ^{2}+\gamma_{n}r_{n}%
\end{align*}
where $C>0$ is a constant independent of $n,$ $\gamma_{n}=2\sigma_{0}%
(1-\beta_{n})\alpha_{n}$ and
$$\ r_{n}=\frac{1}{\sigma_{0}}\left(  \langle
Ty_{n}-q^{\ast},f(q^{\ast})-F(q^{\ast})\rangle+C\alpha_{n}\right)  .$$
Using
the estimate (\ref{HK}), we obtain $\lim\sup r_{n}\leq0;$ hence, by applying
Lemma \ref{Lem1} we conclude as previously that the sequence $\{y_{n}\}$
converges strongly to $q^{\ast}.$ This completes the proof of Theorem
\ref{The2}.
\end{proof}

\section{The study of the limit case $\mu=\alpha$}

Throughout this section, we assume that $\mu=\alpha.$ In this limit case, the operator $F-f$
is monotone but not necessary strongly monotone; so the uniqueness of the
solution of the variational inequality problem (VIP) is no long assured. Let us 
assume that (VIP) has at least one solution. We denote by $S_{VIP}$
the set of the solutions of (VIP). The following theorem provides a method to
approximate a particular element of the set $S_{VIP}.$

\begin{theorem}
Assume that the sequences $\{\alpha_{n}\},~\{\beta_{n}\}$ and $\{e_{n}\}$
satisfy the same assumptions as in Theorem. Then, for every $\varepsilon>0$
and $x_{0}\in Q$, the sequence $\{x_{n}^{\varepsilon}\}$ defined by the
recursive formula%
\[
x_{n+1}^{\varepsilon}=\beta_{n}x_{n}^{\varepsilon}+(1-\beta_{n})P_{Q}\left(
\alpha_{n}f(x_{n}^{\varepsilon})+((1-\alpha_{n}\varepsilon)I-\alpha
_{n}F)Tx_{n}^{\varepsilon}+e_{n}\right)  ,~n\geq0,
\]
converges strongly to $q^{\varepsilon}$ the unique solution of the variational
inequality problem%
\begin{equation}
\text{Find }q\in C\text{ such that }\langle F(q)+\varepsilon q-f(q),x-q\rangle
\geq0\text{ for all }x\in C.\tag{VIP$_\varepsilon$}%
\end{equation}
Moreover, the set $S_{VIP}$ is closed and convex and $q^{\varepsilon}$
converges strongly as $\varepsilon\rightarrow0$ to the nearest element of
$S_{VIP}$ to the origin.
\end{theorem}

\begin{proof}
For every $\varepsilon>0,$ the operator $F_{\varepsilon}:=F+\varepsilon I$ is
$\mu+\varepsilon$ strongly monotone and $\kappa+\varepsilon$ Lipschitizian.
Since $\mu+\varepsilon>\alpha,$ the first part of this theorem follows
immediately from Theorem \ref{The2}.

Let $\delta_{C}:C\rightarrow\mathcal{H}$ be the indicator function associated
to the closed, convex and nonempty subset $C.$ We recall that $\delta_{C\text{
}}$ is given by%
\[
\delta_{C}(x)=\left\{
\begin{array}
[l]{ll}%
0 & ,x\in C\\
+\infty & ,x\notin C
\end{array}
\right.
\]
It is well-known that  $\delta_{C}$  is a proper, lower semi-continuous, and convex function. Hence, its sub-gradient $\partial\delta_{C}$ is a maximal and monotone operator with domain equal to $C$.  We recall that, for every $x\in C$,
\[
\partial\delta_{C}(x)=\{u\in\mathcal{H}:~\langle u,y-x\rangle\leq0\}.
\]
It is then easily seen that the set $S_{VIP}$ is equal to $A^{-1}(0)$ the
set of zeros of the operator $A:=F-f+\delta_{C}.$ From \cite{Roc}, the operator $A$ is maximal and monotone;
therefore, $S_{VIP}$ is a closed and convex subset of $\mathcal{H}.$ On the
other hand, the unique solution $q^{\varepsilon}$ of (VIP$_{\varepsilon}$)
satisfies the relation
\[
-\left(  F(q^{\varepsilon})+\varepsilon q^{\varepsilon}-f(q^{\varepsilon
})\right)  \in\delta_{C}(q^{\varepsilon}),
\]
which is equivalent to
\[
0\in q^{\varepsilon}+\frac{1}{\varepsilon}A(q^{\varepsilon}),
\]
since $\frac{1}{\varepsilon}\delta_{C}(q^{\varepsilon})=\delta_{C}%
(q^{\varepsilon}).$ Therefore, $q^{\varepsilon}=J_{\frac{1}{\varepsilon}}(0),$
where for every $\lambda>0$, $J_{\lambda}=(I+\lambda A)^{-1}$ is the resolvant
of A (for more details, see the pioneer paper \cite{MIN} of Minty). Hence,
from the following lemma due to Bruck \cite{BRU} and Morosanu \cite{MOR}
\begin{lemma}
Let $A:D(A)\subset\mathcal{H}\rightarrow2^{\mathcal{H}}$ be a maximal monotone
operator with $A^{-1}(0)\neq\varnothing.$ Then for any $u\in\mathcal{H},$
$(I+tA)^{-1}u\rightarrow P_{A^{-1}(0)}(u)$ as $t\rightarrow+\infty.$
\end{lemma}
we deduce that $q^{\varepsilon}$ converges
strongly as $\varepsilon\rightarrow0$ to $P_{A^{-1}(0)}=P_{S_{VIP}}(0)$ which
is the element of $S_{VIP}$ with minimal norm.
\end{proof}

\begin{remark}
From the previous theorem, we expect, but we don't yet have the justification,
that under some appropriate assumptions on the real sequences \ $\{\alpha
_{n}\},~\{\beta_{n}\}$ and $\{\varepsilon_{n}\}$,  the sequences $\{x_{n}\}$
generated by the iterative process
\[
x_{n+1}=\beta_{n}x_{n}+(1-\beta_{n})P_{Q}\left(  \alpha_{n}f(x_{n}%
)+((1-\alpha_{n}\varepsilon_{n})I-\alpha_{n}F)Tx_{n}\right)  ,~n\geq0,
\]
where $x_{0}$ is an arbitrary element of $Q$, converge strongly in
$\mathcal{H}$ to $u^{\ast}=P_{S_{VIP}}(0).$ Let us notice that Reich and Xu in
\ \cite{RXU} had raised a similar open question related to the constrained least
squares problem
\end{remark}

%\section{Acknowledgment}
%This work was supported by the Deanship of Scientific Research, King %Faisal University, Saudi Arabia [grant number NA00044].

\end{document}